\documentclass[12pt]{amsart}

\usepackage{amssymb}

\def\aa{{\mathcal A}_{>Z}}
\def\o1{\Omega_1}
\def\cm{{\mathcal M}}
\def\mz{\cm_Z}
\def\mzs{\cm^\ast_Z}
\def\ce{{\mathcal E}}
\def\ff{{\mathbb F}}
\def\a2{{\mathcal A}_{\geq 2}}
\def\ap{{\mathcal A}_p}

\def\de{\Delta}

\def\wh{{\widetilde{H}}}

\def\zz{{\mathbb Z}}
\def\cp{{\mathcal P}}
\def\cs{{\mathcal S}}
\def\nn{{\mathbb N}}

\def\ax{x^\ast}
\def\bx{\bar{x}}
\def\ay{y^\ast}
\def\by{\bar{y}}
\def\la{\langle}
\def\ra{\rangle}

\newtheorem{thm}{Theorem}[section]
\newtheorem{lemma}[thm]{Lemma}
\newtheorem{cor}[thm]{Corollary}
\newtheorem{prop}[thm]{Proposition}

\newtheorem{hyp}[thm]{Hypothesis}

\newtheorem{definition}[thm]{Definition}

\begin{document}

\title{Truncated Quillen complexes of $p$-groups}

\author{Francesco Fumagalli}
\address{Dipartimento di Matematica "Ulisse Dini",viale Morgagni 67/a, 50134 Firenze (FI), ITALY}
\email{fumagalli@math.unifi.it}
\author[Shareshian]{John Shareshian$^1$}
\address{Department of Mathematics, Washington University, St. Louis, MO 63130, USA}
\thanks{$^{1}$Supported in part by NSF Grants
 DMS-0902142 and DMS-1202337}
\email{shareshi@math.wustl.edu}

\maketitle

\begin{abstract}
Let $p$ be an odd prime and let $P$ be a $p$-group.  We examine the order complex of the poset of elementary abelian subgroups of $P$ having order at least $p^2$.  S. Bouc and J. Th\'evenaz showed that this complex has the homotopy type of a wedge of spheres.  We show that, for each nonnegative integer $l$, the number of spheres of dimension $l$ in this wedge is controlled by the number of extraspecial subgroups $X$ of $P$ having order $p^{2l+3}$ and satisfying $\Omega_1(C_P(X))=Z(X)$.  We go on to provide a negative answer to a question raised by Bouc and Th\'evenaz concerning restrictions on the homology groups of the given complex.
\end{abstract}

\section{Introduction}

Since the appearance of the seminal papers \cite{Br1,Br2} by K. S. Brown and \cite{Qu} by D. Quillen, group theorists, topologists and combinatorialists have studied connections between the algebraic structure of a group $G$ and the topology of the order complexes of various posets of $p$-subgroups of $G$.  An excellent and extensive description of such activity appears in the book \cite{Sm}.  The complex that has received the most attention is the {\it Quillen complex} $\de\ap(G)$, whose faces are chains of nontrivial elementary abelian subgroups of $G$.  (Relevant terms used in this introduction will be defined in Section \ref{defs}.)

When $P$ is a finite $p$-group, $\de\ap(P)$ is contractible.  So, one cannot learn much about $P$ from the topology of $\de\ap(P)$.  However, S. Bouc and J. Th\'evenaz observed in \cite{BoTh} that if one removes from $\ap(P)$ all subgroups of $P$ having order $p$, things become more interesting.

\begin{definition} Let $p$ be a prime and let $P$ be a finite $p$-group.  We define $\a2(P)$ to be the set of elementary abelian subgroups of $P$ having order at least $p^2$, ordered by inclusion.
\end{definition}

The main result in \cite{BoTh} is as follows

\begin{thm} \label{bt}
Let $p$ be a prime and let $P$ be a finite $p$-group.  The order complex $\de\a2(P)$ has the homotopy type of a wedge of spheres.
\end{thm}

It is natural to ask how many spheres of each dimension appear in the wedge described in Theorem \ref{bt}.  Equivalently, one can ask for the rank of each homology group of $\de\a2(P)$.  Bouc and Th\'evenaz observed that if $P$ is a $3$-group of order at most $3^6$ then $\de\a2(P)$ has at most one nontrivial reduced homology group.  They observed also that if $P$ s a $2$-group of order at most $2^9$ then $\de\a2(P)$ has at most two nontrivial homology groups, and if $\wh_i(\de\a2(P))$ and $\wh_j(\de\a2(P))$ are nontrivial then $|i-j| \leq 1$.  They asked whether the phenomena just described persist for larger $p$-groups.  A partial positive answer to their question was given by D. Bornand in \cite{Bor}.

\begin{thm}[\cite{Bor}, Corollary 4.12]  \label{born}
Let $p$ be a prime and let $P$ be a finite $p$-group.  Assume $[P,P]$ is cyclic.  If $p$ is odd then $\de\a2(P)$ has at most one nontrivial reduced homology group.  If $p=2$ then $\de\a2(P)$ has at most two nontrivial homology groups, and if $\de\a2(P)$ has two nontrivial homology groups then these groups appear in consecutive dimensions.
\end{thm}

Our main results are Theorem \ref{main1} and Corollaries \ref{main2},  \ref{main3}  and \ref{main4} below.  The point of Corollary \ref{main2} is that determining the homotopy type of $\de\a2(P)$ is the same as enumerating certain subgroups of $P$.  Corollaries \ref{main3} and \ref{main4} provide negative answers to the question of Bouc and Th\'evenaz.

\begin{thm} \label{main1}
Let $p$ be an odd prime and let $P$ be a noncyclic finite $p$-group.  Let $\ce(P)$ be the set of subgroups $X \leq P$ such that
\begin{itemize}
\item $X$ is extraspecial of exponent $p$, and
\item $\Omega_1(C_P(X))=Z(X)$.
\end{itemize}
Then $$\de\a2(P) \simeq \bigvee_{X \in \ce(P)}\de\a2(X).$$
\end{thm}

We prove Theorem \ref{main1} in Section \ref{pr1}.  

When $X$ is extraspecial of exponent $p$, the homotopy type of $\de\a2(X)$ is known.

\begin{prop}[See \cite{Bor}, Proposition 4.7] \label{extra}
Let $p$ be an odd prime and let $X$ be an extraspecial group of exponent $p$ and order $p^{2n+1}$.  Then $\de\a2(X)$ has the homotopy type of a wedge of $p^{n^2}$ spheres of dimension $n-1$.
\end{prop}

Note that when $P$ is a cyclic $p$-group, $\de\a2(P)=\{\emptyset\}$.  By convention, (the geometric realization of) $\de\a2(P)$ is thus the $(-1)$-sphere $S^{-1}$. Combining this fact with Theorem \ref{main1} and Proposition \ref{extra}, we get a new proof of Theorem \ref{bt} when $p$ is odd.  In addition, we get the following result.

\begin{cor} \label{main2}
Let $p$ be an odd prime and let $P$ be a finite $p$-group.  For each positive integer $n$, let $a_n(P)$ be the number of subgroups $X \leq P$ such that \begin{itemize} \item $X$ is extraspecial of exponent $p$ and order $p^{2n+1}$,  and \item $\Omega_1(C_P(X))=Z(X)$.  \end{itemize} Then, for each nonnegative integer $l$, $$\wh_l(\de\a2(P)) \cong \zz^{a_{l+1}(P)p^{(l+1)^2}}.$$
\end{cor}

Applying Theorem \ref{main1} to certain split extensions of extraspecial $p$-groups by automorphisms of order $p$, we prove the following result in Section \ref{pr3}.

\begin{cor} \label{main3}
Let $t$ be nonnegative integer and let $k$ be a positive integer.  For each prime $p>2k+3$ there exists a group $P$ of order $p^{2(t+k+2)}$ such that the only nontrivial homology groups of $\de\a2(P)$ are  $\wh_t(\de\a2(P))$ and $\wh_{k+t}(\de\a2(P))$.
\end{cor}

Further constructions using central products, also found in Section \ref{pr3}, yield the following result.

\begin{cor} \label{main4}
Let $\Omega$ be the smallest collection of subsets of the set $\nn_0$ of nonnegative integers satisfying \begin{enumerate} \item if $I \subseteq \nn_0$ and $|I| \leq 2$, then $I \in \Omega$, and \item if $I,J \in \Omega$, then $$1+I+J:=\{1+i+j:i \in I,j \in J\} \in \Omega.$$ \end{enumerate} For each $I \in \Omega$, there exists an integer $N(I)$ such that for every prime $p>N(I)$, there exists a $p$-group $P$ of exponent $p$ satisfying $$\wh_i(\de\a2(P)) \neq 0 \mbox{ if and only if } i \in I.$$ 
\end{cor}

Applying condition (2) of Corollary \ref{main4} repeatedly to the set $\{0,1\}$, we see that, for each $n \in \nn_0$, the set $\{n,n+1,\ldots,2n,2n+1\}$ lies in $\Omega$.  It follows that for each $m \in \nn_0$, there exist some odd prime $p$ and some $p$-group $P$ such that $\de\a2(P)$ has exactly $m$ nontrivial reduced homology groups.  It remains to be seen whether, for each finite subset $I$ of $\nn_0$, there exist a prime $p$ and a $p$-group $P$ such that $\wh_i(\de\a2(P))$ is nontrivial if and only if $i \in I$.  Also of interest is whether the restriction $p>N(I)$ can be removed from Corollary \ref{main4} or from any stronger result of a similar nature.

\noindent
{\bf Acknowledgment:}
During the preparation of this paper, the first author enjoyed the hospitality and the financial support of Washington University in Saint Louis. He thanks this institution warmly.

\section{Definitions, notation and preliminary results} \label{defs}

\subsection{Topology}
The objects and ideas from topology that we will use are well known and each is discussed in at least one of \cite{Hat}, \cite{Bj} and \cite{Wa}.  

For a partially ordered set $\cp$, the {\it order complex} $\de\cp$ is the abstract simplicial complex whose $k$-dimensional faces are all chains $x_0<\cdots <x_k$ of length $k$ from $\cp$.  We make no distinction between an abstract simplicial complex $\de$ and an arbitrary geometric realization of $\de$, as all such realizations have the same homeomorphism type.  

We write $X \simeq Y$ to indicate that topological spaces $X,Y$ have the same homotopy type.

Let $X_1,\ldots,X_k$ be nonempty, pairwise disjoint topological spaces and pick $x_i \in X_i$ for each $i \in [k]:=\{1,\ldots,k\}$.  The {\it wedge} $\bigvee_{i=1}^{k}(X_i,x_i)$ is obtained from the (disjoint) union $\bigcup_{i=1}^{k}X_i$ by identifying all the $x_i$.  That is, $\bigvee_{i=1}^{k}(X_i,x_i)$ is the quotient space $\bigcup_{i=1}^{k}X_i/\sim$, where $\sim$ is the equivalence relation whose elements are $(x_i,x_j)$ for all $1 \leq i<j \leq k$.   

If, for each $i \in [k]$, the connected components of $X_i$ are pairwise homotopy equivalent, then the homotopy type of $\bigvee_{i=1}^{k}(X_i,x_i)$ does not depend on the choice of the $x_i$ and we write $\bigvee_{i=1}^{k}X_i$ for any wedge of the $X_i$.  In particular, if each $X_i$ is a sphere then the {\it wedge of spheres} $\bigvee_{i=1}^{k}X_i$ is well defined.  For each $l \geq 0$, there is an isomorphism of reduced homology groups $$\wh_l(\bigvee_{i=1}^kX_i) \cong \bigoplus_{i=1}^k \wh_l(X_i).$$  In particular, if each $X_i$ is a sphere then the homotopy type of $\bigvee_{i=1}^{k}X_i$ is uniquely determined by its reduced homology, as asserted implicitly in the introduction.

If $\de$ is a simplicial complex and $\Gamma$ is a contractible subcomplex of $\de$ then $\de$ is homotopy equivalent to the quotient space $\de/\Gamma$.  Therefore, if $\de_1,\de_2$ are simplicial complexes such that $\de_1 \cap \de_2$ is contractible then $\de_1 \cup \de_2$ is homotopy equivalent with a wedge $\de_1 \vee \de_2$.  In particular, if $\de_1$ and $\de_2$ are simplicial complexes such that $\de_1$, $\de_2$ and $\de_1 \cap \de_2$ are contractible, then $\de_1 \cup \de_2$ is also contractible.  It follows by induction on $k$ that if $\de_1,\ldots,\de_k$ are simplicial complexes such that $\bigcap_{i \in I}\de_i$ is contractible for each nonempty $I \subseteq [k]$ then $\bigcup_{i=1}^{k}\de_i$ is contractible.  Now we can derive the following lemma, which is key in our proof of Theorem \ref{main1}.  A result that is essentially the same as this lemma was proved earlier by C. Kratzer and J. Th\'evenaz (see \cite[Lemme 2.8]{KrTh}).

\begin{lemma} \label{ulem}
Say $\de_1,\ldots,\de_k$ are simplicial complexes such that $\bigcap_{i \in I}\de_i$ is contractible for each $I \subseteq [k]$ satisfying $|I| \geq 2$.  Then there exist $x_1,\ldots,x_k$ such that $x_i \in \de_i$ for each $i \in [k]$ and $$\bigcup_{i=1}^{k}\de_i \simeq \bigvee_{i=1}^{k}(\de_i,x_i).$$
\end{lemma}

\begin{proof}
Set $\de:=\bigcup_{i=1}^{k} \de_i$ and $\Gamma:=\bigcup_{1 \leq i<j\leq k}(\de_i \cap \de_j)$. Then $\Gamma$ is contractible.  Therefore, $\de \simeq \de/\Gamma$.  For each $i \in [k]$, set $\Gamma_i:=\bigcup_{j \neq i}(\de_i \cap \de_j)$.  Then $\Gamma_i$ is contractible.  Therefore, $\de_i \simeq \de_i/\Gamma_i$.  Each $\de_i/\Gamma_i$ can be realized as a CW-complex with one $0$-cell $y_i$ corresponding to the subcomplex $\Gamma_i$ along with one $d$-cell for each $d$-dimensional face in $\de_i \setminus \Gamma_i$.  By construction, $\de/\Gamma=\bigvee_{i=1}^{k}(\de_i/\Gamma_i,y_i)$.  To get the wedge described in the lemma, we can take $x_i$ to be any point in $\Gamma_i$. 
\end{proof}

By convention, the wedge of an empty collection of spaces is a point.

\subsection{Group theory} 
We use standard group theoretic notation as can be found in \cite{As,Gor,Suz}.  All groups are assumed to be finite.  For a group $G$ and $x,y \in G$, we write $[x,y]$ for the commutator $x^{-1}y^{-1}xy$.  For $A,B \subseteq G$, we write $[A,B]$ for the subgroup of $G$ generated by all commutators $[a,b]$ with $a \in A$ and $b \in B$.  For $x,y \in G$, we write $x^y$ for $y^{-1}xy$.  We record here the following well known commutator formulas  

\begin{equation} \label{ce1} 
[uv,w]= [u,w]^v[v,w], 
\end{equation}
\begin{equation} \label{ce2}
[u,vw]=[u,w][u,v]^w
\end{equation}
for all $u,v,w \in G$.

Let $p$ be a prime and let $P$ be a $p$-group.  We say that $P$ has {\it exponent} $p$ if every $g \in P$ satisfies $g^p=1$.  A nontrivial group $P$ is {\it elementary abelian} if $P$ is abelian of exponent $p$, in which case $P$ admits a vector space structure over the field $\ff_p$ of order $p$.  The group $P$ is {\it extraspecial} if its center $Z(P)$ has order $p$ and $P/Z(P)$ is elementary abelian.  An extraspecial $p$-group has order $p^{2n+1}$ for some positive integer $n$ (see for example \cite[(23.10)]{As}).  The subgroup $\Omega_1(P) \leq P$ is, by definition, generated by all elements of order $p$ in $P$.  As is standard, we define the terms of the lower central series of $P$ as $\gamma_1(P)=P$ and $\gamma_n(P)=[P,\gamma_{n-1}(P)]$ for each positive integer $n$.  The {\it nilpotence class} of $P$ is the smallest $n$ such that $\gamma_{n+1}(P)=1$.  Finally, $P$ is a {\it central product} of subgroups $Q,R$ if $P=QR$ and $[Q,R]=1$.  Given groups $Q,R$ with subgroups $S \leq Z(Q),T \leq Z(R)$ and an isomorphism $\phi:S \rightarrow T$, one constructs a central product of $Q$ and $R$ as $(Q \times R)/\{(s^{-1},\phi(s)):s \in S\}$.

We will use the following results, the first of which is well known and appears as \cite[(23.8)]{As} and the second of which is due to P. Hall (see \cite[Section 4]{Ha1}, \cite[Introduction]{Ha2}, both of these papers can be found in \cite{Ha3}).

\begin{lemma} \label{es}
Say $X \leq P$ is extraspecial and $[P,X] \leq Z(X)$.  Then $P=XC_P(X)$
\end{lemma}

\begin{thm}[P. Hall] \label{nc}
Let $p$ be a prime and let $P$ be a $p$-group.  If the nilpotence class of $P$ is less than $p$ and $P=\Omega_1(P)$, then $P$ has exponent $p$.  
\end{thm}

\section{The proof of Theorem \ref{main1}} \label{pr1}

Throughout this section, $p$ is an odd prime and $P$ is a $p$-group.  
%We prove that $P$ is not a counterexample to Theorem \ref{main1}.

\subsection{Some useful elementary abelian sections}

Let $Z \leq \Omega_1(Z(P))$.  (While we begin our discussion  with no other condition on $Z$, soon we will turn to the case where $|Z|=p$.) Define $$\mz(P):=\{X \leq P:Z<X=\Omega_1(X) \mbox{ and } [X,X] \leq Z\}.$$ Note that if $Z \neq \Omega_1(P)$ then $\mz(P)$ is nonempty, since $\langle Z,x \rangle \in \mz(P)$ for each $x$ of order $p$ in $P \setminus Z$.  We consider $\mz(P)$ to be partially ordered by inclusion and define $$\mzs(P):=\{X \in \mz(P):X \mbox{ is maximal in $\mz(P)$}\}.$$

Each $X \in \mz(P)$ has nilpotence class at most two.  By Lemma \ref{nc}, each $X \in \mz(P)$ has exponent $p$.

The following technical lemma will be used both in the proof of Theorem \ref{main1} and in the construction of examples proving Corollaries \ref{main3} and \ref{main4} in Section \ref{pr3}.

\begin{lemma} \label{techlem}
Assume $P$ is a central product of proper subgroups $P_1,P_2$ such that at least one $P_i$ has exponent $p$.
% such that \begin{itemize} \item each $P_i$ strictly contains $P_1 \cap P_2$, and 
%\item each $P_i=\Omega_1(P_i)$, and 
%\item at least one $P_i$ has exponent $p$. \end{itemize}  
Set $Z=P_1 \cap P_2$.  Then $$\mzs(P)=\{X_1X_2:X_i \in \mzs(P_i) \mbox{ for } i=1,2\}.$$
\end{lemma}

\begin{proof}
We show first that if $X_i \in \mz(P_i)$ for $i=1,2$ then $X:=X_1X_2 \in \mz(P)$.  It follows from commutator identities (\ref{ce1}) and (\ref{ce2}) that if $g_i,h_i \in P_i$ for $i=1,2$ then \begin{equation} \label{comid} [g_1g_2,h_1h_2]=[g_1,h_1][g_2,h_2]. \end{equation}  Therefore, $$[X,X]=[X_1,X_1][X_2,X_2] \leq Z.$$  Also, $X=\Omega_1(X)$, since $X_i=\Omega_1(X_i)$ for $i=1,2$.  Note also that each $X_i$ lies in $\mz(P)$.

Next let $T \in \mz(P)$.  There are surjective homomorphisms \begin{itemize} \item $\phi:P_1 \times P_2 \rightarrow P$, $\phi((p_1,p_2))=p_1p_2$, \item $\pi_1:P_1 \times P_2 \rightarrow P_1$, $\pi_1((p_1,p_2))=p_1$, and \item $\pi_2:P_1 \times P_2 \rightarrow P_2$, $\pi_2((p_1,p_2))=p_2$. \end{itemize}  For $i=1,2$, set $$T^i:=\pi_i(\phi^{-1}(T)).$$
We claim that, for $i=1,2$, either  $T^i \in \mz(P_i)$ or $T^i=Z$.  The truth of this claim implies the truth of the Lemma.  Indeed, assuming the claim, $T \leq T^1T^2 \in \mz(P)$. It follows that every element of $\mzs(P)$ is of the form $X_1X_2$ with $X_i \in \mzs(P_i)$.  On the other hand, say $X=X_1X_2$ with each $X_i \in \mzs(P_i)$.  If $X \leq T \in \mz(P)$,  then $X_i \leq T^i \in \mz(P_i)$ for $i=1,2$.  Therefore, each $T^i=X_i$  and $T=X$.

Now we prove our claim.  We assume without loss of generality that $i=1$.  Note that $Z  \leq T^1$, since $(z,1) \in \phi^{-1}(Z)$ for each $z \in Z$.  Let $g_1,h_1 \in T^1$.  There exist $g_2,h_2 \in P_2$ such that $g_1g_2$ and $h_1h_2$ are in $T$.   Now $$[g_1g_2,h_1h_2] \in Z \leq P_2,$$ since $[T,T] \leq Z$.  It follows from (\ref{comid}) that $[g_1,h_1] \in P_2$, since both $[g_1g_2,h_1h_2]$ and $[g_2,h_2]$ are in $P_2$.  Therefore, $[g_1,h_1] \in P_1 \cap P_2=Z$ and $[T^1,T^1] \leq Z$.

It remains to show that $T^1=\Omega_1(T^1)$.  Assume for contradiction that $T^1$ has an element $g_1$ of order $p^2$.  There is some $g_2 \in T^2$ such that $g_1g_2 \in T$.  By Lemma \ref{nc}, $(g_1g_2)^p=1$.  As $P_1$ does not have exponent $p$, it must be the case that $g_2^p=1$.  It follows now from $[P_1,P_2]=1$ that $$1=(g_1g_2)^p=g_1^pg_2^p=g_1^p,$$  which gives the desired contradiction and completes our proof.
\end{proof}

\subsection{Proof of Theorem \ref{main1}}

Let us recall a basic assumption under which we are working.

\begin{hyp} \label{bashyp}
The $p$-group $P$ (of odd order) is not cyclic.
\end{hyp}

Our goal is to show that $P$ is not a counterexample to Theorem \ref{main1}.

Assume first that the center $Z(P)$ is not cyclic.  As shown in \cite{BoTh} by applying the Quillen fiber lemma to the map $A \mapsto A\Omega_1(Z(P))$ on $\a2(P)$, $\de\a2(P)$ is contractible.  (This is a key technique in the study of subgroup complexes that is used first, to our knowledge, in \cite{Qu} and described also in \cite[Definition 3.3.1 and Lemma 3.3.3]{Sm}.)  Let $X \leq P$ be extraspecial.  Then $\Omega_1(C_P(X)) \neq Z(X)$, since $Z(X)$ is cyclic.  So, Theorem \ref{main1} holds when $Z(P)$ is not cyclic.  We proceed under the following assumption.

\begin{hyp} \label{assume}
$Z(P)$ is cyclic.
\end{hyp}

Let $Z=\Omega_1(Z(P))$.  Then $Z$ is cyclic of order $p$.  We define $$\aa(P):=\{H \in \a2(P):Z<H\}.$$ As shown in \cite{BoTh} by applying the Quillen fiber lemma to the map $A \mapsto AZ$ on $\a2(P)$,  \begin{equation} \label{a2az} \de\aa(P) \simeq \de\a2(P). \end{equation}  (Again, this is a basic and key technique that is used in \cite{Qu} and appears in \cite[Proposition 3.1.12(2)]{Sm}.)    So, we work from now on with $\aa(P)$ in place of $\a2(P)$.

\begin{lemma} \label{cont}
Let $X \in \mz(P)$.  If $X$ is not extraspecial, then $\de\aa(X)$ is contractible.
\end{lemma}

\begin{proof}
Note first that $X$ is extraspecial if and only if $Z(X)=Z$.  If $Z(X) \neq Z$, then $\Omega_1(Z(X) \in \a2(X)$.  Now we get that $\de\aa(X)$ is contractible by applying the Quillen fiber lemma to the map $A \mapsto AZ(X)$ on $\aa(X)$, as discussed above.
\end{proof}

Note that it follows from Proposition \ref{extra} and Lemma \ref{cont} that the wedge $\bigvee_{X \in \mzs(P)}\de\aa(X)$ is well defined.

\begin{lemma} \label{maxex}
Assume $X \in \mz(P)$ is extraspecial.  Then $X \in \mzs(P)$ if and only if $\Omega_1(C_P(X))=Z$.
\end{lemma}

\begin{proof}
Say $\Omega_1(C_P(X)) \neq  Z$.  Pick $Y<\Omega_1(C_P(X))$ such that $|Y|=p$ and $Y \neq Z$.  Then $X<XY \in \mz(P)$, since $XY \cong X \times Y$.  Therefore, $X \not\in \mzs(P)$.  Conversely, say $X \not\in \mzs(P)$.  There is some $T$ such that $X<T \in \mz(P)$.  Then $X \lhd T$ and $[T,X] \leq Z=Z(X)$, since $[T,T]=Z<X$.  By Lemma \ref{es}, $T=XC_T(X)$.  There is some $g \in C_T(X) \setminus Z$ with $|g|=p$, since $T$ has exponent $p$ and $T \neq X$.   Now $g \in \Omega_1(C_P(X)) \setminus Z$.
\end{proof}

Each $A \in \aa(P)$ lies in $\mz(P)$.  The argument used to obtain (\ref{a2az}) above applies to $X \in \mzs(P)$ as well as to $P$.  Therefore, $$\de\aa(P)=\bigcup_{X \in \mzs(P)}\de\aa(X).$$  Moreover, if an extraspecial $p$-group $X$ does not have exponent $p$, then $X \neq \Omega_1(X)$.  It follows now from Lemma \ref{maxex} that 
$\ce(P)$ consists of those members of $\mzs(P)$ that are extraspecial.  With Lemmas \ref{cont} amd \ref{maxex} in hand, we can invoke Lemma \ref{ulem} to prove Theorem \ref{main1} once we show that the intersection of two or more members of $\mzs(P)$ lies in $\mz(P)$ and is not extraspecial. Lemma \ref{mzsint} below says that this condition on intersections does indeed hold and thus completes the proof of Theorem \ref{main1}.
%Now Theorem \ref{main1} will follow from Lemmas \ref{ulem}, \ref{cont} and \ref{maxex} along with Lemma \ref{mzsint} below.

\begin{lemma} \label{nea}
$P$ has a normal elementary abelian subgroup $N$ of order $p^2$.  Each $S \in \mzs(P)$ contains $N$.
\end{lemma}

\begin{proof}
The existence of $N$ is \cite[Lemma 10.11]{GLS}.   Now $[P,N] \lhd P$.   Moreover, $1<[P,N]<N$, since $Z(P)$ is cyclic and $N$ is not and $P$ is nilpotent.  Therefore, $[P,N]=Z$, since $[P,N] \cap Z(P) \neq 1$.  Let $S \in \mzs(P)$.  Then $[S,N] \leq Z$.  It follows from (\ref{ce1}),(\ref{ce2}) that $[SN,SN] \leq Z$.  Therefore, $SN/Z$ is abelian.  Moreover, $SN=\Omega_1(SN)$, since $S=\Omega_1(S)$ and $N=\Omega_1(N)$.  It follows that $SN \in \mz(P)$.  Therefore, $N \leq S$, since $S \in \mzs(P)$.
\end{proof}

\begin{lemma} \label{mzsint}
Let $\cs \subseteq \mzs(P)$ with $|\cs| \geq 2$.  Then $\bigcap_{S \in \cs}S$ lies in $\mz(P)$ and is not extraspecial.
\end{lemma}

\begin{proof} Let $Y= \bigcap_{S \in \cs}S$.  Then $Z \leq Y$ and, for any $S \in \cs$, $[Y,Y] \leq [S,S] \leq Z$. Moreover, $Y=\Omega_1(Y)$, since each $S \in \cs$ has exponent $p$.  By Lemma \ref{nea}, $Y$ has a subgroup of order $p^2$, so $Z<Y$. Therefore, $Y \in \mz(P)$. 

Assume for contradiction that $Y$ is extraspecial.  Set $Q:=\langle \{S:S \in \cs\} \rangle$.  Now $[Y,Y]=Z=Z(Y)$ and it follows that $[S,Y]=Z$ for each $S \in \cs$, since $[S,S] \leq Z$.  By commutator formulas (\ref{ce1}) and (\ref{ce2}), $[Q,Y]=Z$.  By Lemma \ref{es}, $Q=YC_Q(Y)$.  Note that $Y \cap C_Q(Y)=Z$.

Now $$Q=\Omega_1(Q)=Y\Omega_1(C_Q(Y)).$$ Indeed the first equality holds since $S=\Omega_1(S)$ for each $S \in \cs$.  To prove the second equality, we assume $x \in Q$ has order $p$ and write $x=yh$ with $y \in Y$ and $h \in C_Q(Y)$.  Then $$1=x^p=(yh)^p=y^ph^p=h^p,$$ since $Y$ has exponent $p$ and $[Y,h]=1$.  Therefore, $h \in \Omega_1(C_Q(Y))$ and $x \in Y\Omega_1(C_Q(Y))$ as claimed.

We claim that $C_Q(Y)$ is not cyclic.  Indeed, \begin{eqnarray*} C_Q(Y) & = & C_Q(Y) \cap Y \Omega_1(C_Q(Y)) \\ & = & \Omega_1(C_Q(Y))(Y \cap C_Q(Y)) \\ & = &\Omega_1(C_Q(Y))Z \\ & = & \Omega_1(C_Q(Y)), \end{eqnarray*} the second equality following from the Dedekind modular law (see for example \cite[(1.14)]{As}).  Therefore, if $C_Q(Y)$ is cyclic then $C_Q(Y)=Z$ and $Q=Y$.  However, as $|\cs| \geq 2$, we know that $Y \not\in \cs$ and it follows that $Q \neq Y$.

Assume now that $Z(C_Q(Y))$ is cyclic.  We may apply Lemma \ref{nea} to $C_Q(Y)$.  Every normal subgroup of $C_Q(Y)$ is centralized by $Y$ and therefore normal in $Q$.  Thus there is some $N \lhd Q$ such that \begin{itemize} \item $N$ is elementary abelian of order $p^2$, \item $N \not\leq Y$ (as $[Y,N]=1$ and $|Z(Y)|=p$), and \item every $\tilde{S} \in \mzs(C_Q(Y))$ contains $N$. \end{itemize}  By Lemma \ref{techlem}, each $S \in \mzs(Q)$ satisfies $S=Y\tilde{S}$ for some $\tilde{S} \in \mzs(C_Q(Y))$.   The contradiction $$Y<YN \leq \bigcap_{S \in \cs}S=Y$$ now follows, as $\cs \subseteq \mzs(Q)$.

Finally, assume $Z(C_Q(Y))$ is not cyclic.  Let $N=\Omega_1(Z(C_Q(Y))$.  Then $Z<N$ and $N \not\leq Y$.  Moreover, $N \leq Z(Q)$.  Let $S \in \cs$. Then $[S,N]=1$.  By commutator formulas (\ref{ce1}) and (\ref{ce2}), $[SN,SN]=Z$. Moreover, $SN=\Omega_1(SN)$.  It follows that $SN \in \mz(P)$.  Therefore, $N \leq S$, as $S \in \mzs(P)$.  Again  we get the contradiction $$Y<YN \leq \bigcap_{S \in \cs}S=Y.$$ 
\end{proof} 

\section{Examples} \label{pr3}

Here we prove Corollaries \ref{main3} and \ref{main4}.    Our first step is to produce, for every positive integer $m$ and every prime $p>2m+1$, a $p$-group $P$ of order $p^{2m+2}$ and exponent $p$ such that  \begin{equation} \label{ntriv} \{i:\wh_i(\de\a2(P)) \neq 0\}=\{0,m-1\}. \end{equation}   Recall that in Section \ref{pr1} we observed that when $Z(P)$ is cyclic, $\ce(P)$ consists of those $X \in \mzs(P)$ thst are extraspecial.  Thus by Theorem \ref{main1} and Proposition \ref{extra} (or by Corollary \ref{main2}), if $Z(P)$ is cyclic then the set of all $i>0$ such that $\wh_{i-1}$ is nontrival is the set of all $i>0$ such that $\mzs(P)$ contains an extraspecial $p$-group of order $p^{2i+1}$.  So, for our purposes, it is necessary and sufficient to produce $P$ of the desired order $p^{2m+2}$ and exponent $p$ such that \begin{itemize} \item $Z:=Z(P)$ is cyclic, and \item $\mzs(P)$ contains extraspecial groups of orders $p^3$ and $p^{2m+1}$ and no extraspecial groups of other orders. \end{itemize}

Let $X$ be an extraspecial group of exponent $p$ and order $p^{2m+1}$.  Let $Z=Z(X)$ and let $z$ generate $Z$.  There exist generators $x_1,\ldots,x_m,y_1,\ldots,y_m$ for $X$ such that $$[x_i,x_j]=[y_i,y_j]=1$$ for all $i,j \in [m]$ and $$[x_i,y_j]=\left\{ \begin{array}{ll} z & i=j, \\ 1 & i \neq j. \end{array} \right. $$ Identifying $Z$ with $\ff_p$, one gets a nondegenerate, alternating bilinear form $\la \cdot,\cdot \ra$ on $X/Z$, defined by $$\la Zx,Zy \ra:=[x,y]$$ (see for example \cite[(23.10)]{As}).

Set $$B:=\{g \in Aut(X):z^g=z\}.$$

The action of any $g \in B$ on $X$ induces a linear transformation on the $\ff_p$-vector space $X/Z$.  This transformation preserves the form $\la \cdot,\cdot \ra$.  Thus we have a homomorphism $\Phi$ from $B$ to the group $Sp(X/Z)$ of linear transformations preserving this form.  As shown by D.  L. Winter in \cite{Wi}, $\Phi$ is surjective and the kernel of $\Phi$ is the group $Inn(X)$ of inner automorphisms of $X$.  Thus we have a short exact sequence \begin{equation} \label{ses} 1 \rightarrow Inn(X) \rightarrow B \rightarrow Sp(X/Z) \rightarrow 1 \end{equation} of groups.  As noted by R. Griess in the introduction of \cite{Gr}, the sequence (\ref{ses}) splits.

Assume for the moment that there is some $\phi \in Sp(X/Z)$ such that \begin{itemize} \item $\phi$ has order $p$, and \item the unique eigenspace $C_{X/Z}(\phi)$ for $\phi$ is generated by $Zy_1$. \end{itemize}  (We will see shortly that such a $\phi$ exists whenever $p>2m$.) There is some $g \in B$ such that $|g|=p$ and $\Phi(g)=\phi$, since (\ref{ses}) splits.  

We form the semidirect product $P:=\la g \ra X$.  The group $P$ has order $p^{2m+2}$.  So, the nilpotence class of $P$ is at most $2m+1$.  Assuming $p>2m+1$, we see that $P$ has exponent $p$ by Theorem \ref{nc}, as $P=\Omega_1(P)$.

The group $P/Z$ is not abelian, since $C_{X/Z}(\phi) \neq X/Z$.  Therefore, $X \in \mzs(P)$, as $X$ is maximal in $P$.

There is some integer $j$ with $0 \leq j <p$ such that $[g,y_1]=z^j$, since $Zy_1 \in C_{X/Z}(\phi)$.  By (\ref{ce1}), $[x_1^{1-j},y_1]=z^{1-j}$ and $[gx_1^{1-j},y_1]=z$.   Therefore, the subgroup $S:=\langle gx_1^{1-j},y_1 \rangle \leq P$ is nonabelian of order $p^3$.  It follows that $S$ is extraspecial and lies in $\mz(P)$.  We claim that $S \in \mzs(P)$.  Indeed, say $S \leq T \in \mz(P)$.  Then $T/Z$ is contained in $C_{P/Z}(S/Z)$.  It follows that $T/Z \cap X/Z=\la Zy_1 \ra$, since $\Phi(gx_1^{1-j})=\phi$.  Now $|T/Z| \leq p^2$, since $[P/Z:X/Z]=p$.  It follows that $|T| \leq p^3=|S|$ and $T=S$ as claimed.

Let $T \in \mz(P)$.  We claim that either $T \leq X$ or $|T| \leq p^3$.  Indeed, assume $T \not\leq X$.  Then $T$ contains some $gh$, with $h \in X$, and $T/Z$ centralizes $Zgh$.  Now $C_{X/Z}(Zgh)=\la Zy_1 \ra$, since $\Phi(gh)=\phi$.  Therefore, $|T/Z| \leq p^2$ and $|T| \leq p^3$ as claimed.

We see now that every extraspecial group in $\mzs(P)$ other than $X$ has order $p^3$, since no extraspecial group has order $p^2$.   Next, we produce the desired linear transformation $\phi$, thereby proving the existence of groups $P$ satisfying (\ref{ntriv}). 

Let $V=X/Z$.  We use additive notation for our group operation on $V$ and write $\bx$ for $Zx \in V$.  Now $V$ has basis $\bx_1,\ldots,\bx_m,\by_1,\ldots,\by_m$ and our alternating form $\la \cdot,\cdot \ra$ is given by $$\la \bx_i,\bx_j \ra=\la\by_i,\by_j\ra=0$$ for all $i,j$, and $$\la \bx_i,\by_j \ra=-\la \by_j,\bx_i \ra=\left\{ \begin{array}{ll} 1 & i=j, \\ 0 & i \neq j. \end{array} \right. $$

Define the linear transformation $\phi$ by $$\ax_i:=\bx_i\phi:=(-1)^{m+1-i}\by_m+\sum_{j=i}^{m}(-1)^{j-i}\bx_j$$ for all $i$, $$\ay_i:=\by_i\phi:=\by_i+\by_{i-1}$$ for $2 \leq i \leq m$, and $$\ay_1:=\by_1\phi:=\by_1.$$

To show that $\la v\phi,w\phi \ra=\la v,w \ra$ for all $v,w \in V$, it suffices to examine the cases where $v,w$ lie in our basis.  Certainly $$\la \ay_k,\ay_l \ra=0=\la \by_k,\by_l \ra$$ for all $k,l$.  Also, \begin{eqnarray*} \la \ax_k,\ax_l \ra & = & \la (-1)^{m+1-k}\by_m+\sum_{s=k}^{m}(-1)^{s-k}\bx_s,(-1)^{m+1-l}\by_m+\sum_{t=l}^{m}(-1)^{t-l}\bx_t \ra \\ & = & \la (-1)^{m-k}(\bx_m-\by_m),(-1)^{m-l}(\bx_m-\by_m) \ra \\ & = & 0 \\ & = & \la \bx_k,\bx_l \ra. \end{eqnarray*}

For arbitrary $k$ and $1<l \leq m$,  \begin{eqnarray*} \la \ax_k,\ay_l \ra & = & \sum_{s=k}^{m}(-1)^{s-k}\la \bx_s,\by_l \ra+\sum_{t=k}^{m}(-1)^{t-k}\la \bx_t,\by_{l-1} \ra \\ & = & \left\{ \begin{array}{ll} 0+0 & k>l \\ 1+0 & k=l \\ (-1)^{l-k}+(-1)^{l-1-k} & k<l \end{array} \right.\\ & = & \la \bx_k,\by_l \ra. \end{eqnarray*}

Finally, \begin{eqnarray*} \la \ax_k,\ay_1 \ra & = & \sum_{s=k}^{m}(-1)^{s-k}\la \bx_s,\by_1 \ra \\ & = & \left\{\begin{array}{ll} 1 & k=1 \\ 0 & k>1 \end{array} \right. \\ & = & \la \bx_k,\by_1 \ra. \end{eqnarray*}

With respect to the ordering $x_1,\ldots,x_m,y_m,\ldots,y_1$ of our basis (note the $y_i$ appear in reverse order), the matrix of $\phi$ is upper triangular with $1$ on the diagonal.  Therefore, $(\phi-1)^{2m}=0$.  So, if $p>2m$ then $$0=(\phi-1)^p=\phi^p-1$$ (the second equality following from the binomial theorem) and $\phi$ has order $p$.

Now we compute $C_V(\phi)$, the unique eigenspace of $\phi$.  Pick
\[
v=\sum_{i=1}^m\alpha_i \bx_i+\sum_{j=1}^{m}\beta_j \by_j,
\]
an arbitrary element of $V$.  Write
\[
v\phi=\sum_{i=1}^{m}\gamma_i\bx_i+\sum_{j=1}^{m}\delta_j\by_j.
\]
Then $v \in C_V(\phi)$ if and only if $\alpha_i=\gamma_i$ and $\beta_j=\delta_j$ for all $i,j$.  We calculate \begin{equation} \label{xeq} \gamma_i=\sum_{s=1}^{i}(-1)^{i-s}\alpha_s \end{equation} for $1 \leq i \leq m$, \begin{equation} \label{yeq} \delta_j=\beta_j+\beta_{j+1} \end{equation} for $1 \leq j<m$, and \begin{equation} \label{ymeq} \delta_m=\beta_m+\sum_{s=1}^{m}(-1)^{m+1-s}\alpha_s. \end{equation}

Say $v \in C_V(\phi)$.  From (\ref{xeq}) we conclude that $\alpha_i=0$ for $1 \leq i \leq m-1$.  With this conclusion in hand, we obtain from (\ref{ymeq}) that $\alpha_m=0$.  From (\ref{yeq}) we conclude that $\beta_j=0$ for $2 \leq j \leq m$.  Certainly $\by_1 \in C_V(\phi)$.  We see now that $$C_V(\phi)=\la\by_1\ra$$ as desired.

We observe now that, with $P$ a central product of $P_1$ and $P_2$ as in Lemma \ref{techlem} and $X_i \in \mzs(P_i)$ for each $i$, the product $X_1X_2$ is extraspecial if and only if both $X_i$ are extraspecial.  Moreover, if each $X_i$ is extraspecial of order $p^{2m_i+1}$ (and therefore $\de\a2(X_i)$ has reduced homology concentrated in degree $m_i-1$) then $X_1X_2$ has order $p^{2(m_1+m_2)+1}$ (and $\de\a2(X_1X_2)$ has reduced homology concentrated in degree $m_1+m_2-1$).  Corollary \ref{main3} is obtained from Lemma \ref{techlem} by taking $P_1$ be a group of exponent $p$ and order $p^{2k+4}$ with nontrivial reduced homology concentrated in degrees $0$ and $k$, as constructed above (with $k=m+1$) and taking $P_2$ to be extraspecial of order $p^{2t+1}$ and exponent $p$.  Corollary \ref{main4} follows from Lemma \ref{techlem}, Proposition \ref{extra} and Corollaries \ref{main2} and \ref{main3}.   Indeed, Proposition \ref{extra} and Corollary \ref{main3} show that $I \in \Omega$ in the cases $|I|=1$ and $|I|=2$, respectively.  Moreover, $\emptyset \in \Omega$, since $\a2(P)$ is contractible if $P$ is elementary abelian of order at least $p^2$.  Now assume that $P_1$ and $P_2$ both have cyclic center and exponent $p$, and that the set of orders of elements of $\mzs(P_i)$ is $\{2j+1:j \in I_i\}$ for $i=1,2$. Let $P$ be a central product of $P_1$ and $P_2$. Then, by Lemma \ref{techlem}, the set of orders of elements of $\mzs(P)$ is $\{2(i_1+i_2)+1:i_1 \in I_1,i_2 \in I_2\}$.  Part (2) of Corollary \ref{main4} now follows from Corollary \ref{main2}.


\begin{thebibliography}{BoTh}

\bibitem[As]{As} M. Aschbacher, {\it Finite Group Theory}, Cambridge University Press, 2000.

\bibitem[Bj]{Bj} A. Bj\"orner, Topological methods. Handbook of combinatorics, Vol. 1, 2, 1819Ð1872, Elsevier, Amsterdam, 1995. 

\bibitem[Bor]{Bor} D. Bornand, Elementary abelian subgroups in $p$-groups with a cyclic derived subgroup. J. Algebra 335 (2011), 301-318.

\bibitem[BoTh]{BoTh} S. Bouc and J. Th\'evenaz, The poset of elementary abelian subgroups of rank at least $2$.  Monogr. Enseign. Math. 40 (2008) 41-45.

\bibitem[Br1]{Br1} K. S. Brown, Euler characteristics of discrete groups and $G$-spaces. 
Invent. Math. 27 (1974), 229-264.

\bibitem[Br2]{Br2} K. S. Brown, Euler characteristics of groups: the $p$-fractional part. Invent. Math. 29 (1975), no. 1, 1-5. 

\bibitem[Gor]{Gor} D. Gorenstein, Finite groups. Second edition. Chelsea Publishing Co., New York, 1980.

\bibitem[GLS]{GLS} D. Gorenstein, R. Lyons and R. Solomon, The classification of the finite simple groups. Number 2. Part I. Chapter G. (English summary) 
General group theory. Mathematical Surveys and Monographs, 40.2. American Mathematical Society, Providence, RI, 1996. 

\bibitem[Gr]{Gr} R. L. Griess, Jr., Automorphisms of extra special groups and nonvanishing degree 2 cohomology. Pacific J. Math. 48 (1973), 403-422.

\bibitem[Ha1]{Ha1} P. Hall, A contribution to the theory of groups of prime power order.  Proc. London Math. Soc. (2) 36 (1934), 29-95.

\bibitem[Ha2]{Ha2} P. Hall, On a theorem of Frobenius.  Proc. London Math. Soc. (2) 40 (1935), 468-501.

\bibitem[Ha3]{Ha3}  P. Hall, The collected works of Philip Hall. Compiled and with a preface by K. W. Gruenberg and J. E. Roseblade. With an obituary by Roseblade. Oxford Science Publications. The Clarendon Press, Oxford University Press, New York, 1988.

\bibitem[Hat]{Hat} A. Hatcher, Algebraic topology. Cambridge University Press, Cambridge, 2002.

\bibitem[KrTh]{KrTh} C. Kratzer and J. Th\'evenaz, Type d'homotopie des treillis et treillis des sous-groupes d'un groupe fini. Comment. Math. Helv. 60 (1985), no. 1, 85-106. 

\bibitem[Qu]{Qu} D. Quillen, Homotopy properties of the poset of nontrivial $p$-subgroups of a group. Adv. in Math. 28 (1978), no. 2, 101-128.

\bibitem[Sm]{Sm} S. D. Smith, Subgroup complexes. Mathematical Surveys and Monographs, 179. American Mathematical Society, Providence, RI, 2011.

\bibitem[Suz]{Suz} M. Suzuki, Group theory. II. Translated from the Japanese. Grundlehren der Mathematischen Wissenschaften [Fundamental Principles of Mathematical Sciences], 248. Springer-Verlag, New York, 1986.

\bibitem[Wa]{Wa} M. L. Wachs, Poset topology: tools and applications. Geometric combinatorics, 497Ð615, IAS/Park City Math. Ser., 13, Amer. Math. Soc., Providence, RI, 2007.

\bibitem[Wi]{Wi} D. L. Winter, The automorphism group of an extraspecial $p$-group. Rocky Mountain J. Math. 2 (1972), no. 2, 159-168.

\end{thebibliography}
\end{document}